\documentclass[12pt]{amsart}
\usepackage{latexsym, amsmath, amssymb, amsthm, mathrsfs}
\usepackage[alwaysadjust]{paralist}
\usepackage{caption}
\usepackage{subcaption}
\usepackage[T1]{fontenc}
\usepackage{lmodern}
\usepackage[backref,colorlinks=true,linkcolor=blue,urlcolor=blue, citecolor=blue]%
  {hyperref}
\usepackage[shortalphabetic]{amsrefs}
\usepackage[all]{xy}
\usepackage[T1]{fontenc}
\usepackage{mathptmx}
\usepackage{microtype}
\usepackage{framed}
\usepackage{tikz}
\usepackage{graphicx}

\usepackage[centering, includeheadfoot, hmargin=1in, vmargin=1in,
  headheight=30.4pt]{geometry}
\newtheorem{lemma}{Lemma}[section]
\newtheorem{theorem}[lemma]{Theorem}
\newtheorem{theorem*}{Theorem}
\newtheorem{corollary}[lemma]{Corollary}
\newtheorem{proposition}[lemma]{Proposition}

\theoremstyle{definition}

\newtheorem{remark}[lemma]{Remark}

\renewcommand{\theequation}%
{\arabic{section}.\arabic{lemma}.\arabic{equation}}

\renewcommand{\AA}{\ensuremath{\mathbb{A}}} 
\newcommand{\CC}{\ensuremath{\mathbb{C}}} 
 
\newcommand{\PP}{\ensuremath{\mathbb{P}}}

\newcommand{\sI}{\ensuremath{\mathcal{I}}}        
\newcommand{\sJ}{\ensuremath{\kern -2pt \mathscr{J}\kern -2pt}} 
\newcommand{\sO}{\ensuremath{\mathscr{O}}}

\renewcommand{\geq}{\geqslant}
\renewcommand{\leq}{\leqslant}

\newcommand{\equ}{\ensuremath{\,=\,}}
\newcommand{\deq}{\ensuremath{\stackrel{\textrm{def}}{=}}}

\newcommand{\dsubseteq}{\ensuremath{\,\subseteq\,}}
\newcommand{\brac}[1]{\ensuremath{\left( #1\right)}}
\newcommand{\dgeq}{\ensuremath{\,\geq\,}}
\newcommand{\dleq}{\ensuremath{\,\leq\,}}
\newcommand{\xa}{\ensuremath{x_1^{a_1}\dots x_n^{a_n}}}
\begin{document}

\title[Containment relations]{An elementary approach to containment relations between symbolic and ordinary powers of certain monomial ideals}

\author[R.~W.~Keane]{Ryan W. Keane}
\author[A.~K\" uronya]{Alex K\" uronya}
\address{Alex K\"uronya, Budapest University of Technology and Economics, Department of Algebra, Egry J\'ozsef u. 1., H-111 Budapest, Hungary}
\address{Goethe Universit\"at Frankfurt, Institut f\"ur Mathematik, Robert-Mayer-Str. 6-10, D-60325 Frankfurt am Main, Germany}
\email{{\tt alex.kuronya@math.bme.hu}}

\author[E.~McMahon]{Elise McMahon}



%

\maketitle


\section{Introduction}

The purpose of this note is to find an elemenary explanation of a surprising result of Ein--Lazarsfeld--Smith \cite{ELS} and 
Hochster--Huneke \cite{HH} on the containment
between symbolic and ordinary powers of ideals in simple cases. This line of research has been very active ever since, see for instance \cites{BC,HaH,DST}
and the references therein, by now the literature on this topic is quite extensive. 

By `elementary' we refer to arguments that among others do not make use of resolution of singularities and multiplier ideals nor tight closure methods. 
Let us quickly recall the statement \cite{ELS}: let $X$ be a smooth projective variety  of dimension $n$, $\sI\subseteq\sO_X$ a non-zero sheaf of radical ideals  with zero scheme $Z\subseteq X$; if every irreducible component 
of $Z$ has codimension at least $e$, then
\[
 \sI^{(me)}_Z \dsubseteq \sI^m_Z
\]
for all $m\geq 1$. 

Our  goal is to reprove this assertion in the case of points in projective spaces (as asked in \cite{PAGII}*{Example 11.3.5})
without recurring to deep methods of algebraic  geometry. Instead of working with subsets of  projective space, 
we will concentrate on the affine cones over them;  our aim hence becomes to understand symbolic and ordinary powers  ideals
of sets of line through the origin. 

We will end up reducing the general case to a study of the ideals 
\[
I_{2,n} \deq \brac{x_ix_j\mid 1\leq i<j\leq n}\dsubseteq k[x_1,\dots,x_n]
\]
defining the union of coordinate axes in $\AA^n_{k}$. We work over an arbitrary field $k$. 
%

\begin{theorem*}[Corollary~\ref{thm:main}]
Let $\Sigma\subseteq \PP^{n-1}$ be a set of $n$ points not lying in a hyperplane. Then 
\[
 \sI_{\Sigma}^{(\lceil (2-\frac{2}{n})m \rceil) } \subseteq \sI^m_{\Sigma}
\]
for all positive integers $m$. 

If $n=3$, then the same statement holds for three distinct points in arbitrary position. 
\end{theorem*}

\noindent
\textbf{Acknowledgements.} This work is the result of an `Elective Undergraduate Research' class at the Budapest Semesters in Mathematics 
in Spring 2014. After our proofs were completed, Brian Harbourne and Tomasz Szemberg have pointed out closely related work in the area 
(see Remark~\ref{rmk:Brian},  Remark~\ref{rmk:Brian2}, and Remark~\ref{rmk:Tomek}, respectively, for more precise statements and references) 
for which we are grateful.


\section{Proof of the main result}

The key idea is to describe the primary decomposition of $I_{2,n}$.

\begin{proposition}\label{prop:main technical}
With notation as above, the expression
\begin{equation}\label{primary}
 I_{2,n}^m \equ   \brac{\bigcap_{i=1}^{n}(x_1,\dots,\widehat{x_i},\dots,x_n)^m} \cap (x_1,\dots,x_n)^{2m}
\end{equation}
is a primary decomposition of $I_{2,n}^m$  for all positive integers $m$. 
\end{proposition}
\begin{proof} 
Being an irredundant intersection of primary ideals, the right-hand side of (\ref{primary}) will be a primary decomposition of 
$I_{2,n}^m$ once equality holds.

We assume that $n\geq 3$, the other cases being trivial. 
Since  $I_{2,n} = \cap_{i=1}^{n}(x_1,\dots,\widehat{x_i},\dots,x_n)$ and $I_{2,n}\subseteq I_{1,n}^2$, we have 
\[
  I_{2,n}^m \dsubseteq   \brac{\bigcap_{i=1}^{n}(x_1,\dots,\widehat{x_i},\dots,x_n)^m} \cap (x_1,\dots,x_n)^{2m}\ .
\]
For the reverse inclusion, note that both sides are monomial ideals, therefore it will suffice to check the required containment
for monomials by \cite{HeHi}*{Theorem 1.1.2}

A monomial $\xa$ is an element of the right-hand side of (\ref{primary}) if and only if the inequalities
 \begin{eqnarray}\label{ineq1} 
 a_1 + \ldots + \widehat{a_i} + \ldots + a_n  & \dgeq m & \text{ for all $1\leq i\leq n$,} \\
\label{ineq2} \sum_{i=1}^{n} a_i & \dgeq & 2m  
\end{eqnarray} 
hold. We will show by induction on $m$ that under these conditions $\xa\in I_{2,n}^m$. For the base case $m=1$ of the induction, one has 
\[
 I_{2,n} \equ \bigcap_{i=1}^{n}(x_1,\dots,\widehat{x_i},\dots,x_n) \dsubseteq  I_{1,n}\ ,
\]
as $I_{2,n}$ is the radical ideal of the union of all coordinate axes in $\AA^n$. 

Let us now assume  that (\ref{primary}) holds for all integers less than $m$, and let 
\[
 \xa \in \brac{\bigcap_{i=1}^{n}(x_1,\dots,\widehat{x_i},\dots,x_n)^m} \cap (x_1,\dots,x_n)^{2m}
\]
be an arbitrary monomial, without loss of generality we will suppose that $a_1\geq\ldots\geq a_n$. It follows from  the inequality
(\ref{ineq2}) that the total degree of $\xa$ is at least $2m$. 

Let us first consider the case $a_1=\ldots=a_n=a$ and $\sum_{i=1}^{n}a_i=2m$. If $n$ is even, then 
\[
\xa \equ x_1^a\ldots x_n^a \equ  (x_1x_2)^a\ldots (x_{n-1}x_n)^a \in I_{2,n}^m\ ,
\]
if $n$ is odd, then $a=2b$ for some natural number $b$ (since $an = 2m$), and 
\[
\xa \equ  x_1^a\ldots x_n^a \equ (x_1x_2)^b(x_2x_3)^b\ldots (x_{n-1}x_n)^b(x_nx_1)^b  \in I_{2,n}^m\ .
\]
From now on we will assume that not all exponents $a_i$ are the same, and keep the assumption that $a_1\geq\ldots\geq a_n$. 
Also, without loss of generality we will suppose that $A\deq  \sum_{i=1}^{n}a_i =  \deg \xa = 2m$: if it were strictly higher, then either 
$A-a_i> m$ for all $1\leq i\leq n$, and we can  divide $\xa$ by some variable and still obtain a monomial in the same ideal
(and then possibly reorder the variables to preserve  the decreasing exponents), or, 
if $A-a_1=m$ (which is the smallest one among the $A-a_i$'s), then $a_1=A-m>m$, hence replacing $\xa$ by $x_1^{a_1-1}x_2^{a_2}\ldots x_n^{a_n}$ will again do the trick. 

We proceed by  a greedy algorithm: we intend to show that 
\[
 x_1^{a_1-1}x_2^{a_2-1}x_3^{a_2}\ldots x_n^{a_n} \in I_{2,n}^{m-1}
\]
by verifying that the exponent vector satifsfies that system (\ref{primary}) for $m-1$. Granting this, it follows that
\[
\xa \equ (x_1x_2)\cdot x_1^{a_1-1}x_2^{a_2-1}x_3^{a_3}\ldots x_n^{a_n} \in I_{2,n}^m\ ,
\]
as required. 

We need to check that
\[
 A-(a_i-d_i) \dgeq m-1 \ \ \text{for all $1\leq i\leq n$,}
\]
where $d_1=d_2=1$, and $d_i=0$  for all $i\geq 3$, and 
\[
 A - 2 \geq 2(m-1)\ .
\]
Of these, the last equality is immediate, we will deal with the rest. As $a_1\geq a_2\geq\ldots\geq a_n$, we have
\[
 m\dleq A-a_1 \dleq \ldots \dleq  A-a_n\ ,
\]
hence we are done whenever either $i=1,2$, or $i\geq 3$ and $a_i<a_1$. 

Let us suppose that this is not the case, and there exists an index $3\leq i\leq n$ such that $a_1=a_i$. By the ordering of the $a_i$'s
this automatically means $a_1=\ldots = a_j = \ldots=a_i$ for all $1\leq j\leq i$. Note that even in this case, $A-a_i-2\geq m-1$ if 
$A-a_i>m$, hence aiming at a contradiction we can assume $A-a_i=m$. On the other hand, 
\[
 2m \equ  A \equ m+a_i\ ,
\]
therefore $a_1=\ldots = a_i =  m$, and consequently, 
\[
 2m \equ A \dgeq a_1+\dots+a_i \dgeq 3m\ ,
\]
a contradiction. 
\end{proof}

\begin{lemma}\label{symbolicpowers}
With notation as above, 
\[
 I_{2,n}^{(m)} \equ \bigcap_{i=1}^{n} (x_1,\dots,\widehat{x_i},\dots,x_n)^m \ .
\]
\end{lemma}
\begin{proof}
Follows quickly from the facts that symbolic and ordinary powers of $(x_1,\dots,x_i)$ agree, and the fact that 
\[
 I_{2,n} \equ \bigcap_{i=1}^{n} (x_1,\dots,\widehat{x_i},\dots,x_n)
\]
is an irredundant intersection where all prime ideals on the right-hand side are minimal. 
\end{proof}

\begin{remark}\label{remark:symbolic-inequalities}
 Observe that for a monomial $\xa\in k[x_1,\dots,x_n]$, $\xa\in I_{2,n}^{(m)}$ if and only if $A-a_i\geq m$ for all $1\leq i\leq n$.
\end{remark}

\begin{proposition}
 With notation as above, 
 \[
  I_{2,n}^{(\lceil (2-\frac{2}{n})m \rceil )} \dsubseteq I_{2,n}^{m}
 \]
for all $m\geq 1$.
\end{proposition}

\begin{proof}
Assume  $d\geq \lceil (2-\frac{2}{n})m \rceil$.  For a monomial $\xa\in I_{2,n}^{(d)}$, 
Remark~\ref{remark:symbolic-inequalities} implies $A-a_i\geq d $ for all $1\leq i\leq n$, hence  
\[
 nA - A \dgeq nd\ ,
\]
and 
\[
 A \dgeq \frac{n}{n-1}d \dgeq \frac{n}{n-1}\cdot (2-\frac{2}{n})m \dgeq 2m\ ,
\]
and $\xa\in I_{2,n}^m$ by Proposition~\ref{prop:main technical} as required. 
\end{proof}

\begin{remark}
It appears that in the concrete case the bound from \cite{ELS}, which is $2m$, where the $2$ stands for the codimension of the union of coordinate axes), 
is far from optimal even for $n=3$. 
\end{remark}

\begin{theorem}\label{thm:coordinate lines to arbitrary}
Let $X\subseteq \AA^n_k$ be a union of $n$ lines through the origin not lying in a hyperplane, where $k$ is an arbitrary field. Then
\[
 I(X)^{(\lceil (2-\frac{2}{n})m \rceil )} \dsubseteq I(X)^m 
\]
for all $m\geq 1$. 

If $n=3$ then the statement holds for any  three distinct lines through the origin.  
\end{theorem}
\begin{proof}
 Let $\phi\colon \AA^n\to \AA^n$ be an invertible  linear map taking the coordinates axes in $\AA^n$ to the lines in $X$. Then $\phi^*$
 induces an automorphism of $\CC[x_1,\dots,x_n]$, under which $I(X)^{(\lceil (2-\frac{2}{n})m \rceil )}$ corresponds to 
 $I_{2,n}^{(\lceil (2-\frac{2}{n})m \rceil )}$, and $I(X)^m$ corresponds to $I_{2,n}^m$. 

 As far as the case $n=3$ is concerned, observe that three lines through the origin that lie in a plane form a complete intersection subvariety, 
 and as such, $m$\textsuperscript{th} symbolic and ordinary powers of its vanishing ideal will agree for any given $m\geq 1$. 
 \end{proof}

\begin{corollary}\label{thm:main}
Let $\Sigma\subseteq\PP^{n-1}_{k}$ be a set of $n$ points not lying in a hyperplane. Then 
\[
 \sI_{\Sigma}^{(\lceil (2-\frac{2}{n})m \rceil )} \subseteq \sI^m_{\Sigma}
\]
for all positive integers $m$. 

If $n=3$, then the result holds for any three distinct points. 
\end{corollary}

In the rest of the article  we outline connections to closely related work in the area. For terminology not used in our paper we refer 
to \cites{BC, BC2, PS}.

\begin{remark}\label{rmk:Brian}
  As $n$ general points in $\PP^{n-1}$ form a star configuration, the initial degree and the regularity of the associated ideal agree, hence 
  \cite{BC2}*{Corollary 1.2} applies. Combined with \cite{PS}*{Lemma 8.4.7 (c)},  this gives a way to establish a bound similar in nature to 
Corollary~\ref{thm:main}. 
\end{remark}

\begin{remark}\label{rmk:Brian2}
In \cite{PS}*{Example 8.4.5} a bound  stronger then the one in \cite{ELS} is shown for all monomial ideals. Nevertheless, in the particular case we treated
the bound of \cite{PS}*{Example 8.4.5} is weaker than Corollary~\ref{thm:main}.  
\end{remark}

\begin{remark}\label{rmk:Tomek}
 Multiplier ideals of line arrangements and their relationship to symbolic powers  were studied for instance in \cite{T1} and  \cite{T2}. 
\end{remark}



\raggedright

\end{document}